\documentclass[12pt]{amsart}
\usepackage{amssymb,verbatim,amscd,amsmath,graphicx}
\usepackage{graphicx}
\usepackage{caption}
\usepackage{subcaption}
\usepackage{pdfsync}

\numberwithin{equation}{section}
\newtheorem{theorem}{Theorem}[section]

\newtheorem{proposition}[theorem]{Proposition}
\newtheorem{corollary}[theorem]{Corollary}
\newtheorem{conjecture}[theorem]{Conjecture}

\theoremstyle{definition}
\newtheorem{definition}[theorem]{Definition}
\newtheorem{def-prop}[theorem]{Definition-Proposition}
\newtheorem{obs}[theorem]{Observation}
\newtheorem{remark}[theorem]{Remark}

\DeclareMathOperator{\depth}{depth}

\DeclareMathOperator{\Ass}{Ass}

\newcommand{\mf}{\mathfrak{m}}

\def\1{{\bf 1}}
\def\0{{\bf 0}}

\begin{document}


\title[Persistence property and non-increasing depth]{Squarefree Monomial Ideals that Fail the Persistence Property and Non-increasing Depth}

\author{Huy T\`ai H\`a}
\address{Tulane University \\ Department of Mathematics \\
6823 St. Charles Ave. \\ New Orleans, LA 70118, USA}
\email{tha@tulane.edu}
\urladdr{http://www.math.tulane.edu/$\sim$tai/}

\author{Mengyao Sun}
\address{Tulane University \\ Department of Mathematics \\
6823 St. Charles Ave. \\ New Orleans, LA 70118, USA}
\email{msun@tulane.edu}

\thanks{H\`a is partially supported by the Simons Foundation (grant \#279786).}

\begin{abstract}
In a recent work \cite{KSS}, Kaiser, Stehl\'ik and \v{S}krekovski provide a family of critically 3-chromatic graphs whose expansions do not result in critically 4-chromatic graphs, and thus give counterexamples to a conjecture of Francisco, H\`a and Van Tuyl \cite{FHVT}. The cover ideal of the smallest member of this family also gives a counterexample to the persistence and non-increasing depth properties. In this paper, we show that the cover ideals of \emph{all} members of their family of graphs indeed fail to have the persistence and non-increasing depth properties.
\end{abstract}

\maketitle

\begin{center}
\emph{Dedicate to Professor Ng\^o Vi\^et Trung in honors of his sixtieth birthday}
\end{center}


\section{Introduction} \label{sec.intro}

Let $k$ be a field and let $R = k[x_1, \dots, x_n]$ be a polynomial ring over $k$. Let $I \subseteq R$ be a homogeneous ideal. It is known by Brodmann \cite{Brod1} that the set of associated primes of $I^s$ stabilizes for large $s$, that is, $\Ass(R/I^s) = \Ass(R/I^{s+1})$ for all $s \gg 0$. However, the behavior of these sets can be very strange for small values of $s$. The ideal $I$ is said to have the \emph{persistence property} if
$$\Ass(R/I^s) \subseteq \Ass(R/I^{s+1}) \ \forall \ s \ge 1.$$
It is also known by Brodmann \cite{Brod2} that $\depth (R/I^s)$ takes a constant value for large $s$. The behavior of $\depth(R/I^s)$, for small values of $s$, can also be very complicated. The ideal $I$ is said to have \emph{non-increasing depth} if
$$\depth(R/I^s) \ge \depth(R/I^{s+1}) \ \forall \ s \ge 1.$$

Associated primes and depth of powers of ideals have been extensively investigated in the literature (cf. \cite{BHH, FHVT1, FHVT, FHVT2, HM, HH, HQ, HRV, HV, MMV, M, MV, TT}). Even for monomial ideals, it is difficult to classify which ideals possess the persistence property or non-increasing depth. In this case, when $I$ is a monomial ideal, the two properties are related by the fact that $I$ possesses the persistence property if all monomial localizations of $I$ have non-increasing depth. Herzog and Hibi \cite{HH} gave an example where $\mf = (x_1, \dots, x_n) \in \Ass(R/I^s)$ for small even integers $s$ (whence $\depth(R/I^s) = 0$) and $\mf \not\in \Ass(R/I^s)$ for small odd integers $s$ (whence $\depth(R/I^s) > 0$). Squarefree monomial ideals behave considerably better than monomial ideals in general, and many classes of squarefree monomial ideals were shown to have the persistence property. For instance, {\it edge ideals} of graphs (\cite{MMV}), {\it cover ideals} of perfect graphs, {\it cover ideals} of {\it cliques}, {\it odd holes} and {\it odd antiholes} (\cite{FHVT}), and {\it polymatroidal} ideals (\cite{HRV}). A large class of squarefree monomial ideals with constant depth was constructed in \cite{HV}. 

In an attempt to tackle the persistence property, at least in identifying a large class of squarefree monomial ideals having the persistence property, the first author, together with Francisco and Van Tuyl in \cite{FHVT}, made a graph-theoretic conjecture about {\it expansion} of {\it critically $s$-chromatic} graphs and proved that this conjecture implies the persistence property for the cover ideals of graphs. The converse {\it a priori} is not known to be true. Recently, Kaiser, Stehl\'ik and \v{S}krekovski \cite{KSS} gave a family of counterexamples to this graph-theoretic conjecture. Computational experiment showed that the first member of their family of graphs also gave a counterexample to the persistence property and non-increasing depth for squarefree monomial ideals. In fact, this is the only graph with at most 20 vertices whose cover ideal fails the persistence property (\cite{Wehlau}). The goal of this paper is to prove that all members of this family indeed give counterexamples to the persistence property. As a consequence, they also provide counterexamples to non-increasing depth property.

Let us now describe more specifically our problem and results. Let $G = (V,E)$ be a {\it simple} graph with vertex set $V = \{x_1, \dots, x_n\}$ and edge set $E$. The \emph{expansion} of $G$ at a vertex $x \in V$ is obtained from $G$ by adding a new vertex $y$, an edge $\{x,y\}$, and edges $\{y,z\}$ whenever $\{x,z\}$ is an edge in $G$. For a subset $W \subseteq V$, the \emph{expansion} of $G$ at $W$, denoted by $G[W]$, is obtained by expanding successively at the vertices in $W$. The first author, together with Francisco and Van Tuyl in \cite{FHVT}, made the following conjecture.

\begin{conjecture} \label{conj.criticalgraphs}
Let $G$ be a critically $s$-chromatic graph. Then there exists a subset $W$ of the vertices such that $G[W]$ is critically $(s+1)$-chromatic.
\end{conjecture}

The \emph{cover ideal} of $G = (V,E)$ is defined to be
$$J(G) = \bigcap_{\{x,y\} \in E} (x,y).$$
It was also shown in \cite{FHVT} that if Conjecture \ref{conj.criticalgraphs} holds then the persistence property holds for the cover ideal of any graph. The converse is not known to be true.

\begin{figure}[h!]
\centering
\includegraphics[height=2in]{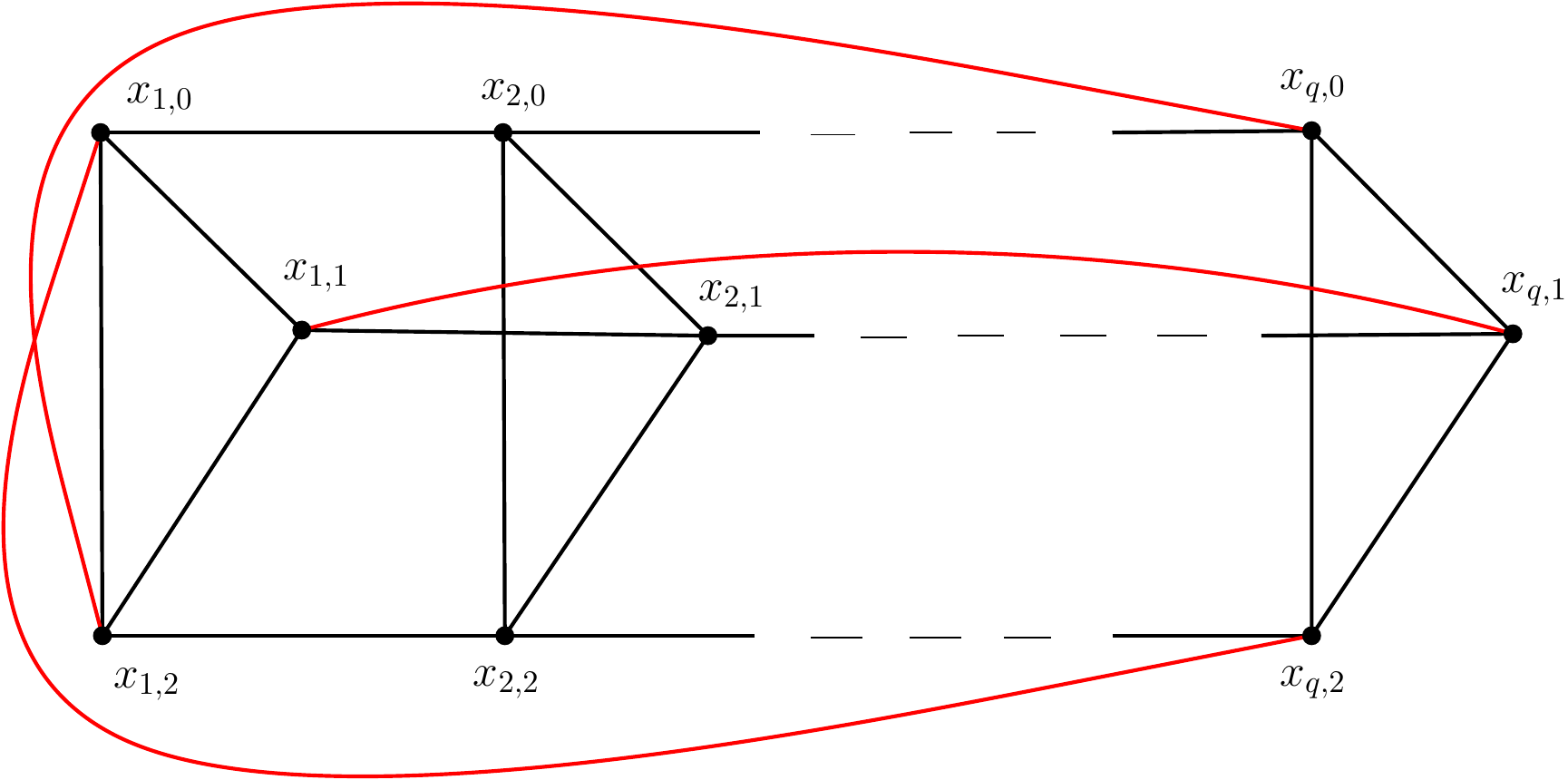}
\caption{The graph $H_q = K_3 \boxtimes P_q$.}\label{fig:graph}
\end{figure}

A family of counterexamples to Conjecture \ref{conj.criticalgraphs} was given by Kaiser, Stehl\'ik and \v{S}krekovski \cite{KSS} as follows. Let $K_3$ denote the complete graph on 3 vertices, and let $P_q$, for $q \ge 4$, denote a {\it path} of length $q-1$. The graph $H_q = K_3 \boxtimes P_q$ is formed by taking $q$ copies of $K_3$ with vertices $\{x_{i,0}, x_{i,1}, x_{i,2}\}, i = 1, \dots, q,$ connecting $x_{i,j}$ and $x_{i+1,j}$ for $i=1, \dots, q-1$ to get 3 paths of length $q-1$, and finally connecting $x_{1,j}$ and $x_{q,2-j}$ for $j=0, 1, 2$ (see Figure \ref{fig:graph}). These graphs were originally constructed by Gallai \cite{Gallai}. One of the interesting properties of $H_q$s is that they embed in the Klein bottle as quadrangulations (see \cite{KSS}).

It was pointed out in \cite{KSS} that when $q = 4$, the cover ideal $J = J(H_4)$ fails the persistence property and non-increasing depth. In particular, if $\mf$ is the maximal homogeneous ideal then $\mf \in \Ass(R/J^3)$ but $\mf \not\in \Ass(R/J^4)$. The main result of this paper is to show that this phenomenon occurs for all $q \ge 4$.

\begin{theorem} \label{thm.intro}
Let $H_q$ be the graph constructed as before. Let $J = J(H_q)$ and let $\mf$ be the maximal homogeneous ideal in the polynomial ring $R = k[x_{i,j} ~|~ i=1, \dots, q, j=0,1,2]$. Then $\mf \in \Ass(R/J^3)$ and $\mf \not\in \Ass(R/J^4)$. As a consequence, $J$ fails to have non-increasing depth.
\end{theorem}


\section{Preliminaries} \label{sec.prel}

In this section, we collect notation and terminology used in the paper. We follow standard texts in the research area \cite{Berge, BrHe, HHbook, MS}.

Let $k$ be a field, let $R = k[x_1, \dots, x_n]$, and let $\mf = (x_1, \dots, x_n)$. Throughout the paper, we shall identify the variables of $R$ with distinct vertices $V = \{x_1, \dots, x_n\}$. A graph $G = (V,E)$ consists of $V$ and a set $E$ of edges connecting pairs of vertices. We shall restrict our attention to \emph{simple} graphs, i.e., graphs without loops nor multiple edges.

\begin{definition} Let $G$ be a simple graph.
\begin{enumerate}
\item The \emph{chromatic number} of a graph $G$, denoted by $\chi(G)$, is the least number of colors to assign to the vertices so that adjacent vertices get different colors.
\item The graph $G$ is said to be \emph{critically $s$-chromatic} if $\chi(G) = s$, and for any vertex $x$ in $G$, $\chi(G \backslash x) < s$.
\end{enumerate}
\end{definition}
Critically $s$-chromatic graphs are also known as \emph{$s$-vertex-critical} graphs. We choose to use the terminology of critically $s$-chromatic graphs to be consistent with \cite{FHVT}, where Conjecture \ref{conj.criticalgraphs} was stated.

A collection of vertices $W \subseteq V$ in $G = (V,E)$ is called a \emph{vertex cover} if for any edge $e \in E$, $W \cap e \not= \emptyset$. A vertex cover is called \emph{minimal} if no proper subset of it is also a vertex cover.

There are various ways to associate squarefree monomial ideals to simple graphs. In this paper, the correspondence that we explore is the cover ideal construction.

\begin{definition} Let $G = (V,E)$ be a simple graph. The \emph{cover ideal} of $G$ is defined to be
$$J(G) = \bigcap_{\{x,y\} \in E} (x,y).$$
\end{definition}
The term \emph{cover} ideal comes from the fact that minimal generators of $J(G)$ correspond to minimal vertex covers in $G$.
This cover ideal construction gives a one-to-one correspondence between simple graphs and unmixed, codimension two squarefree monomial ideals (this construction extends to hypergraphs to give a one-to-one correspondence to all squarefree monomial ideals).

\begin{definition} Let $I \subseteq R$ be an ideal. A prime ideal $P$ is said to be an \emph{associated prime} of $I$ if there exists an element $c \in R$ such that $P = I : c$. The set of all associated primes of $I$ is denoted by $\Ass(R/I)$.
\end{definition}

\begin{definition} Let $M$ be a finitely generated $R$-module.
\begin{enumerate}
\item A sequence of elements $x_1, \dots, x_t \in R$ is called an $M$-regular sequence if $M \not= (x_1, \dots, x_t)M$ and $x_i$ is not a zero-divisor of $M/(x_1, \dots, x_{i-1})M$ for all $i = 1, \dots, t$.
\item The \emph{depth} of $M$, denoted by $\depth(M)$, is the largest length of an $M$-regular sequence in $R$.
\end{enumerate}
\end{definition}

\begin{remark} \label{rmk.depth0}
It is an easy exercise to see that for an ideal $I \subseteq R$, $\depth (R/I) > 0$ if and only if $\mf \not\in \Ass(R/I)$.
\end{remark}

\begin{remark} \label{rmk.Hpq}
The construction of the graph $H_q$ can be generalized to a pair consisting of a path and a complete graph of any size. Indeed, let $P_q$ be a path of length $q-1$ and let $K_p$ be the complete graph of size $p$. We can then construction the graph $H_{p,q} = K_p \boxtimes P_q$ by taking $q$ copies of $K_p$ with vertices $\{x_{i,0}, \dots, x_{i,p-1}\}$, $i = 1, \dots, q$, connecting $x_{i,j}$ to $x_{i+1,j}$ for $i=1, \dots, q-1$ to get $p$ paths of length $q-1$, and finally connecting $x_{1,j}$ to $x_{q,p-1-j}$ for $j = 0, \dots, p-1$. In this construction, $H_q = H_{3,q}$.
\end{remark}


\section{Proof of the main result} \label{sec.proof}

This section is devoted to the proof of our main result, Theorem \ref{thm.intro}. This theorem will be proved as a combination of Propositions \ref{prop.power3} and \ref{prop.power4} and Corollary \ref{cor.depth}. For simplicity of terminology, we call the complete graph $K_3$ on $\{x_{i,0}, x_{i,1}, x_{i,2}\}$ the $i$th triangle in $H_q$. We shall also abuse notation in identifying vertices of $H_q$ and corresponding variables in $R$.

\begin{proposition} \label{prop.power3}
Let $H_q$ be the graph constructed as in the introduction. Let $J = J(H_q)$ and let $\mf$ be the maximal homogeneous ideal in $R = k[x_{i,j} ~|~ i = 1, \dots, q, j = 0,1,2]$. Then $\mf \in \Ass(R/J^3)$.
\end{proposition}

\begin{proof} It was shown in \cite[Proposition 9]{KSS} that $H_q$ is critically 4-chromatic. Thus, it follows from \cite[Corollary 4.5]{FHVT1} that $\mf \in \Ass(R/J^3)$.
\end{proof}

\begin{proposition} \label{prop.power4}
Let $H_q$ be the graph constructed as in the introduction. Let $J = J(H_q)$ and let $\mf$ be the maximal homogeneous ideal in $R = k[x_{i,j} ~|~ i = 1, \dots, q, j = 0,1,2]$. Then $\mf \not\in \Ass(R/J^4)$.
\end{proposition}

\begin{proof} Suppose, by contradiction, that $\mf \in \Ass(R/J^4)$. That is, there exists a monomial $T$ in $R$ such that $T \not\in J^4$ and $J^4 : T = \mf$. Since the generators of $J$ are squarefree, the powers of each variable in minimal generators of $J^4$ are at most 4. This implies that the power of each variable in $T$ is at most 3, i.e., $T$ divides $\big(\prod_{i,j} x_{i,j}\big)^3$. We shall now make a few observations to reduce the number of cases that we need to consider later.

\begin{obs} \label{obs.power3}
$M = \big(\prod_{i,j} x_{i,j}\big)^3 \in J^4$. Indeed, we can write $M = M_1 M_2 M_3 M_4 N$ as follows.
\begin{enumerate}
\item If $q$ is odd then choose $N = \prod_{i=1}^q x_{i,0}$ and
\begin{align*}
M_1 & = \big(\prod_{i < q \text{ odd}} x_{i,0} x_{i,1}\big) \big(\prod_{i \text{ even}} x_{i,1} x_{i,2}\big) (x_{q,0} x_{q,2}) \\
M_2 & = \big(\prod_{i < q \text{ odd}} x_{i,0} x_{i,2}\big) \big(\prod_{i \text{ even}} x_{i,1} x_{i,2}\big) (x_{q,0} x_{q,1}) \\
M_3 & = \big(\prod_{i < q \text{ odd}} x_{i,1} x_{i,2}\big) \big(\prod_{i \text{ even}} x_{i,0} x_{i,1}\big) (x_{q,1} x_{q,2}) \\
M_4 & = \big(\prod_{i < q \text{ odd}} x_{i,1} x_{i,2}\big) \big(\prod_{i \text{ even}} x_{i,0} x_{i,2}\big) (x_{q,1} x_{q,2}).
\end{align*}
\item If $q$ is even then choose $N = \prod_{i=1}^q x_{i,1}$ and
\begin{align*}
M_1 & = \big(\prod_{i \text{ odd}} x_{i,0} x_{i,2}\big) \big(\prod_{i \text{ even}} x_{i,0} x_{i,1}\big) \\
M_2 & = \big(\prod_{i \text{ odd}} x_{i,0} x_{i,2}\big) \big(\prod_{i \text{ even}} x_{i,1} x_{i,2}\big) \\
M_3 & = \big(\prod_{i \text{ odd}} x_{i,0} x_{i,1}\big) \big(\prod_{i \text{ even}} x_{i,0} x_{i,2}\big) \\
M_4 & = \big(\prod_{i \text{ odd}} x_{i,1} x_{i,2}\big) \big(\prod_{i \text{ even}} x_{i,0} x_{i,2}\big).
\end{align*}
\end{enumerate}
It is an easy exercise to verify that $M_1, \dots, M_4$ are vertex covers of $H_q$. Thus, $M \in J^4$. This observation allows us to assume that $T$ strictly divides $M$.
\end{obs}

\begin{obs} \label{obs.totalpower}
For each $i = 1, \dots, q$, the total power of $x_{i,0}, x_{i,1}$ and $x_{i,2}$ in $T$ is at least 8. Indeed, take $k \not= i$, then since $J^4 : T = \mf$, we must have $Tx_{k,0} \in J^4$. It then follows from the fact that generators of $J$ correspond to vertex covers of $H_q$ that $Tx_{k,0}$ can be written as the product of 4 vertex covers of $H_q$. Notice also that to cover the triangle with vertices $\{x_{i,0}, x_{i,1}, x_{i,2}\}$ each vertex cover needs at least two of those 3 vertices. Thus, 4 vertex covers contain in total at least 8 copies of those vertices. This observation and the fact that $T$ divides $M$ allow us to conclude that for each $i = 1, \dots, q$, either all three vertices $\{x_{i,0}, x_{i,1}, x_{i,2}\}$ appear in $T$ each with power exactly 3, or two of them appear in $T$ with power 3 and the third one appears in $T$ with power exactly 2. For simplicity of language, we shall call the total power of $\{x_{i,0}, x_{i,1}, x_{i,2}\}$ in $T$ the {\it power of the $i$th triangle in $T$}.
\end{obs}

\begin{obs} \label{obs.easy}
Suppose that the power of the $i$th triangle in $T$ is at least 8, and we already impose the conditions that 3 among the $M_i$s each has to contain a specific (but distinct) variable in the $i$th triangle. Then we can always distribute the remaining variables of the $i$th triangle from $T$ into the $M_i$s so that each of them indeed covers the edges of the $i$th triangle. To see this, without loss of generality, we may assume that the 3 imposed conditions are $x_{i,0} ~\big|~ M_1, x_{i,1} ~\big|~ M_2$ and $x_{i,2} ~\big|~ M_3$, and assume that $x_{i,1}$ and $x_{i,2}$ appear in $T$ with powers at least 3. This implies that $x_{i,0}$ appears in $T$ with power at least 2, and we can distribute the variables of the $i$th triangle in $T$ into the $M_i$s as follows:
\begin{align*}
x_{i,0} x_{i,1} & \Big|~ M_1 \\
x_{i,1} x_{i,2} & \Big|~ M_2 \\
x_{i,1} x_{i,2} & \Big|~ M_3 \\
x_{i,0} x_{i,2} & \Big|~ M_4.
\end{align*}
\end{obs}

\begin{obs} \label{obs.firstTA}
Re-indexing the vertices of $H_q$ as follows: label $x_{q,0}$ by $x_{1,2}$, label $x_{q,1}$ by $x_{1,1}$, label $x_{q,2}$ by $x_{1,0}$ (notice that we have switched the second indices 0 and 2 in the $q$ triangle and bring it to be the first triangle), and then label $x_{i,j}$ by $x_{i+1,j}$ for all $1 \le i \le q-1$ and $j = 0,1,2$ (i.e., shifting the rest of the triangles one place to the right). We then obtain an isomorphic copy of $H_q$ where the old $q$th triangle becomes the first one. This process can be repeated. Thus, coupled with Observation \ref{obs.power3}, we can assume that the power of the first triangle in $T$ is exactly 8. Without loss of generality, we may further assume that $x_{1,0}$ appears in $T$ with power 2, while $x_{1,1}$ and $x_{1,2}$ appear in $T$ with powers 3.
\end{obs}

\begin{obs} \label{obs.distribution}
Fix an index $i < q-1$ where the power of the $i$th triangle in $T$ is exactly 8, and assume that $x_{i,0}$ appears in $T$ with power 2 (and so, $x_{i,1}$ and $x_{i,2}$ appear in $T$ both with power 3). Since $J^4 : T = \mf$, in particular, we have $Tx_{q,0} \in J^4$. That is, we can write $Tx_{q,0} = M_1 M_2 M_3 M_4$ as the product of 4 elements in $J$, i.e., 4 vertex covers of $H_q$. To distribute $x_{i,0}^2x_{i,1}^3x_{i,2}^3$ into 4 vertex covers, there is only one possibility (up to permutation of the indices of the vertex covers), which is:
\begin{align*}
x_{i,0} x_{i,1} & \Big|~ M_1 \\
x_{i,0} x_{i,2} & \Big|~ M_2 \\
x_{i,1} x_{i,2} & \Big|~ M_3 \\
x_{i,1} x_{i,2} & \Big|~ M_4.
\end{align*}
This distribution of the vertices of the $i$th triangle will impose specific conditions on how the vertices of the $(i+1)$st triangle can be distributed into the 4 vertex covers. Particularly, we must have that $x_{i+1,2} ~\big|~ M_1, x_{i+1,1} ~\big|~ M_2$, and $x_{i+1,0}$ divides both $M_3$ and $M_4$.

If the power of the $(i+1)$st triangle in $T$ is 9 then we can distribute vertices in the $(i+1)$st triangle into the $M_i$s as follows:
\begin{align*}
x_{i,0} x_{i,1} \quad x_{i+1,1} x_{i+1,2} & \Big|~ M_1 \\
x_{i,0} x_{i,2} \quad x_{i+1,1} x_{i+1,2} & \Big|~ M_2 \\
x_{i,1} x_{i,2} \quad x_{i+1,0} x_{i+1,1} & \Big|~ M_3 \\
x_{i,1} x_{i,2} \quad x_{i+1,0} x_{i+1,2} & \Big|~ M_4,
\end{align*}
where the extra copy of $x_{i+1,0}$ could be assigned to either $M_1$ or $M_2$. Now, the only conditions imposed on the $(i+2)$nd triangle are $x_{i+2,2} ~\big|~ M_3, x_{i+2,1} ~\big|~ M_4$, and either $x_{i+2,0} ~\big|~ M_2$ or $x_{i+2,0} ~\big|~ M_1$. It follows from Observation \ref{obs.easy} that the variables of the $(i+2)$nd triangle in $T$ can be distributed into the $M_i$s, and we can think of the $(i+2)$nd triangle as our new starting point (if $i+2 < q$).


If, on the other hand, the power of the $(i+1)$st triangle in $T$ is 8, then we obtain the following possibilities depending on which variable in the $(i+1)$st triangle appears in $T$ with power 2.

\begin{enumerate}
\item If the power of $x_{i+1,0}$ in $T$ is 2 then (up to permuting $M_3$ and $M_4$) we have:
\begin{align*}
x_{i,0} x_{i,1} \quad x_{i+1,1} x_{i+1,2} & \Big|~ M_1 \\
x_{i,0} x_{i,2} \quad x_{i+1,1} x_{i+1,2} & \Big|~ M_2 \\
x_{i,1} x_{i,2} \quad x_{i+1,0} x_{i+1,1} & \Big|~ M_3 \\
x_{i,1} x_{i,2} \quad x_{i+1,0} x_{i+1,2} & \Big|~ M_4.
\end{align*}
\item If the power of $x_{i+1,1}$ in $T$ is 2 then we must be in either of the following cases:
\begin{align*}
x_{i,0} x_{i,1} \quad x_{i+1,0} x_{i+1,2} & \Big|~ M_1 \\
x_{i,0} x_{i,2} \quad x_{i+1,1} x_{i+1,2} & \Big|~ M_2 \\
x_{i,1} x_{i,2} \quad x_{i+1,0} x_{i+1,1} & \Big|~ M_3 \\
x_{i,1} x_{i,2} \quad x_{i+1,0} x_{i+1,2} & \Big|~ M_4;
\end{align*}
or
\begin{align*}
x_{i,0} x_{i,1} \quad x_{i+1,1} x_{i+1,2} & \Big|~ M_1 \\
x_{i,0} x_{i,2} \quad x_{i+1,0} x_{i+1,1} & \Big|~ M_2 \\
x_{i,1} x_{i,2} \quad x_{i+1,0} x_{i+1,2} & \Big|~ M_3 \\
x_{i,1} x_{i,2} \quad x_{i+1,0} x_{i+1,3} & \Big|~ M_4.
\end{align*}
\item If the power of $x_{i+1,2}$ in $T$ is 2 then we must be in either of the following cases:
\begin{align*}
x_{i,0} x_{i,1} \quad x_{i+1,0} x_{i+1,2} & \Big|~ M_1 \\
x_{i,0} x_{i,2} \quad x_{i+1,1} x_{i+1,2} & \Big|~ M_2 \\
x_{i,1} x_{i,2} \quad x_{i+1,0} x_{i+1,1} & \Big|~ M_3 \\
x_{i,1} x_{i,2} \quad x_{i+1,0} x_{i+1,1} & \Big|~ M_4;
\end{align*}
or
\begin{align*}
x_{i,0} x_{i,1} \quad x_{i+1,1} x_{i+1,2} & \Big|~ M_1 \\
x_{i,0} x_{i,2} \quad x_{i+1,0} x_{i+1,1} & \Big|~ M_2 \\
x_{i,1} x_{i,2} \quad x_{i+1,0} x_{i+1,1} & \Big|~ M_3 \\
x_{i,1} x_{i,2} \quad x_{i+1,0} x_{i+1,2}& \Big|~ M_4.
\end{align*}
\end{enumerate}
The upshot of this observation is that we can successively distribute $T$ and $Tx_{q,0}$ (without the use of the extra variable $x_{q,0}$) into 4 vertex covers up to the $(q-1)$st triangle in the same way. At each step, moving from the $i$th triangle to the $(i+1)$st triangle, we might end up with a number of different choices. Moreover, if the power of the $(i+1)$st triangle in $T$ is 9, then we can distribute the vertices in the $i$th and the $(i+1)$st triangles, and consider the $(i+2)$nd triangle as our new starting point to repeat the process. The difference, and what makes $T \not\in J^4$ but $Tx_{q,0} \in J^4$, occurs when we need to cover the $q$th triangle and edges connecting the $q$th and the $1$st triangles (i.e., moving from the $(q-1)$st triangle to the last triangle).
\end{obs}

By making use of Observation \ref{obs.distribution}, we can successively distribute the variables appearing in $T$ into the $M_i$s in the same way as that of $Tx_{q,0}$ such that along the process, $M_i$s cover edges in the first $(q-1)$ triangles. It remain to consider how the variables in the $q$th triangle are distributed. We shall show that a contradiction, either that $T \in J^4$ or that $J^4 : T \not= \mf$, is always resulted in.

Notice that when the power of the $(q-1)$st triangle in $T$ is 9, in our distribution process, a power 8 of this triangle is distributed to the $M_i$s, and there is possibly an {\it extra} copy of a variable left. This possible extra variable can then be assigned to one of the $M_i$s. Our argument will complete by exhausting cases depending on how the vertices in the $(q-1)$st triangle are distributed among the $M_i$s and which vertex is possibly treated as the extra one.

There are 3 choices for the possible extra vertex. For each choice of the possible extra vertex, the cases are considered depending on how the other two copies of this vertex are distributed among 4 vertex covers $M_i$s. Observe that if the possible extra vertex is $x_{q-1,t}$ (where $t = 0, 1$ or $2$, and we identify $x_{i,t}$ with $x_{i,t+3}$) then there are 6 cases to consider by assigning $x_{q-1,t}$ to 2 out of the 4 vertex covers $M_i$s. For example, if $x_{q-1,t}$ is assigned to $M_1$ and $M_2$, then there would be two possibilities depending on how $x_{q-1,t+1}$ and $x_{q-1,t+2}$ are distributed. These possibilities are described by conditions:
\begin{align*}
x_{1,0} x_{1,1} \dots \dots x_{q-1,t} x_{q-1,t+1} & \Big|~ M_1 \\
x_{1,0} x_{1,2} \dots \dots x_{q-1,t} x_{q-1,t+2} & \Big|~ M_2 \\
x_{1,1} x_{1,2} \dots \dots x_{q-1,t+1} x_{q-1,t+2} & \Big|~ M_3 \\
x_{1,1} x_{1,2} \dots \dots x_{q-1,t+1} x_{q-1,t+2} & \Big|~ M_4,
\end{align*}
or
\begin{align*}
x_{1,0} x_{1,1} \dots \dots x_{q-1,t} x_{q-1,t+2} & \Big|~ M_1 \\
x_{1,0} x_{1,2} \dots \dots x_{q-1,t} x_{q-1,t+1} & \Big|~ M_2 \\
x_{1,1} x_{1,2} \dots \dots x_{q-1,t+1} x_{q-1,t+2} & \Big|~ M_3 \\
x_{1,1} x_{1,2} \dots \dots x_{q-1,t+1} x_{q-1,t+2} & \Big|~ M_4.
\end{align*}

This case-by-case analysis is quite tedious, but the 18 cases are mostly similar. Thus, we will go through the argument carefully for one case and leave it to the interested reader to check the details of the remaining cases.

Consider the case where $x_{q-1,0}$ is the possible extra vertex, and the other two copies of $x_{q-1,0}$ are in $M_1$ and $M_2$. There are two possibilities depending on how $x_{q-1,1}$ and $x_{q-1,2}$ were distributed:
\begin{align*}
x_{1,0} x_{1,1} \dots \dots x_{q-1,0} x_{q-1,1} & \Big|~ M_1 \\
x_{1,0} x_{1,2} \dots \dots x_{q-1,0} x_{q-1,2} & \Big|~ M_2 \\
x_{1,1} x_{1,2} \dots \dots x_{q-1,1} x_{q-1,2} & \Big|~ M_3 \\
x_{1,1} x_{1,2} \dots \dots x_{q-1,1} x_{q-1,2} & \Big|~ M_4,
\end{align*}
or
\begin{align*}
x_{1,0} x_{1,1} \dots \dots x_{q-1,0} x_{q-1,2} & \Big|~ M_1 \\
x_{1,0} x_{1,2} \dots \dots x_{q-1,0} x_{q-1,1} & \Big|~ M_2 \\
x_{1,1} x_{1,2} \dots \dots x_{q-1,1} x_{q-1,2} & \Big|~ M_3 \\
x_{1,1} x_{1,2} \dots \dots x_{q-1,1} x_{q-1,2} & \Big|~ M_4.
\end{align*}

If it is the first possibility that occurs, and there is in fact no extra copy of $x_{q-1,0}$ (i.e., the power of the $(q-1)$st triangle in $T$ was exactly 8), then this impose the following conditions on the $q$th triangle: $x_{q,0} x_{q,2} ~\big|~ M_1, M_3, M_4$ and $x_{q,1} ~\big|~ M_2$. This implies that the product of the $M_i$s will use 4 copies of either $x_{q,0}$ or $x_{q,2}$. Thus, $Tx_{q,1} \not\in J^4$. If it is the first possibility but there is an extra copy of $x_{q,0}$ left, then we can distribute this extra copy of $x_{q,0}$ to either $M_3$ or $M_4$, say $M_4$. In this case, the conditions imposed on the $q$th triangle are: $x_{q,0} x_{q,2} ~\big|~ M_1$ and $M_3$, $x_{q,1} ~\big|~ M_2$, and $x_{q,2} ~\big|~ M_4$. Thus, to cover the edges of the $q$th triangle, we must have
\begin{align*}
x_{q,0} x_{q,2} & \Big|~ M_1 \\
x_{q,0} x_{q,1} & \Big|~ M_2 \\
x_{q,0} x_{q,2} & \Big|~ M_3 \\
x_{q,1} x_{q,2} & \Big|~ M_4.
\end{align*}
It follows that if $T$ contains 3 copies of $x_{q,0}$ and $x_{q,2}$ then this distribution shows that $T \in J^4$. Otherwise, if $T$ contains, for instance, only 2 copies of $x_{q,0}$, then since the product of 4 vertex covers, as shown, must contain at least 3 copies of $x_{q,0}$, we have that $Tx_{q,1} \not\in J^4$.

If it is the second possibility and there is no extra copy of $x_{q-1,0}$ then conditions imposed on the $q$th triangle are: $x_{q,0} x_{q,1} ~\big|~ M_1, x_{q,1}x_{q,2} ~\big|~ M_2, x_{q,0}x_{q,2} ~\big|~ M_3$ and $M_4$. Thus, the product of the 4 vertex covers contain at least 3 copies of $x_{q,0}$ and $x_{q,2}$. If $T$ has at least 3 copies of $x_{q,0}$ and $x_{q,2}$ then $T \in J^4$. Otherwise, $Tx_{q,1} \not\in J^4$. If it is the second possibility and there is an extra copy of $x_{q-1,0}$ then we can distribute this extra copy of $x_{q,0}$ to either $M_3$ and $M_4$, say $M_4$. In this case, the conditions imposed on the $q$th triangle are: $x_{q,0} x_{q,1} ~\big|~ M_1, x_{q,1}x_{q,2} ~\big|~ M_2, x_{q,0}x_{q,2} ~\big|~ M_3$ and $x_{q,2} ~\big|~ M_4$. Thus, if $T$ contains at least 3 copies of $x_{q,2}$ then by distributing either $x_{q,0}$ or $x_{q,1}$ to $M_4$, we get that $T \in J^4$. Otherwise, $Tx_{q,0} \not\in J^4$.

For the remaining cases, it can be seen that covering the edges of the $q$th triangle and edges connecting to the first and the $(q-1)$st triangles will impose a number of conditions on how vertices of the $q$th triangle in $T$ can be distributed to the 4 vertex covers $M_i$s. These conditions will fall into one of the following situations.
\begin{enumerate}
\item The conditions do not require $\prod_{i=1}^4 M_i$ to contain any vertex of the $q$th triangle with power more than 2. In this case, we can always distribute the vertices of the $q$th powers in $T$ into the 4 vertex covers $M_i$s in a way to satisfy these conditions. We thus have $T \in J^4$.
\item The conditions require $\prod_{i=1}^4 M_i$ to contain one or two vertices of the $q$th triangle with powers at least 3. If $T$ indeed does contain those vertices with powers at least 3, then we can also distribute the vertices of the $q$th triangle in $T$ into the 4 vertex covers $M_i$ to comply with those condition; we then have $T \in J^4$. If, otherwise, $T$ does not contain those one or two vertices with powers at least 3, then the product of $T$ with the third vertex will not be in $J^4$.
\item The conditions require $\prod_{i=1}^4 M_i$ to contain a vertex of the $q$th triangle with power at least 4. In this case, the product of $T$ and another vertex of the $q$th triangle will not be in $J^4$.
\end{enumerate}
\end{proof}

\begin{corollary} \label{cor.depth}
Let $H_q$ be the graph constructed as in the introduction. Let $J = J(H_q) \subseteq R = k[x_{i,j} ~|~ i = 1, \dots, q, j = 0,1,2]$. Then $J$ fails to have non-increasing depth.
\end{corollary}

\begin{proof} The conclusion is a direct consequence of Remark \ref{rmk.depth0}, Propositions \ref{prop.power3} and \ref{prop.power4}.
\end{proof}


\section{Other constructions} \label{sec.others}

A natural generalization of the graphs $H_q$s are those of $H_{p,q}$s as constructed in Remark \ref{rmk.Hpq}. We end the paper by showing that those graphs $H_{p,q}$ do not give counterexamples to Conjecture \ref{conj.criticalgraphs}. In fact, we shall show that $H_{p,q}$, for $p > 3$ are not critical graphs.

\begin{theorem} Let $p, q \ge 4$ and let $H_{p,q}$ be constructed as in Remark \ref{rmk.Hpq}. Then, $\chi(H_{p,q}) = p$, but $H_{p,q}$ is not critical $p$-chromatic.
\end{theorem}

\begin{proof} Clearly, any graph containing a complete subgraph of size $p$ has the chromatic number at least $p$. Thus, it suffices to show that $H_{p,q}$ can be colored using $p$ colors (and since $H_{p,q}$ contains more than one copies of $K_p$, this will also imply that $H_{p,q}$ is not critical $p$-chromatic). Indeed, we can assign $p$ colors to the vertices of $H_{p,q}$ as follows. We shall identify colors congruent modulo $p$.

\begin{figure}[h!]
\centering
\includegraphics[width=3.5in]{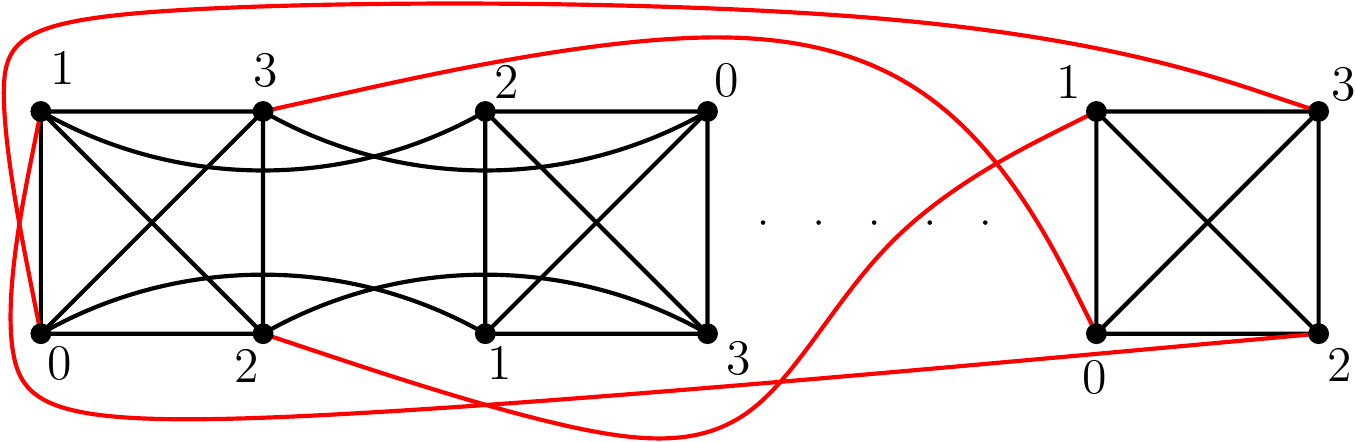}
\caption{A 4-coloring for $H_{4,q}$ when $q$ is odd.}\label{fig:graphcolor4odd}
\end{figure}

\noindent{\it Case 1: $p$ is even and $q$ is odd}. For $1 \le i \le q$ and $i$ is odd, assign to $x_{i,j}$ color $j$ for all $j = 0, \dots, p-1$. For $1 \le i \le q$ and $i$ is even, assign to $x_{i,j}$ color $j+1$, for $j=0, \dots, p-1$. It is easy to see that the vertices on each copy of $K_p$ get different colors. Also, on the $i$th and $(i+1)$st copies of $K_p$, since the parity of $i$ and $i+1$ are different, adjacent vertices $x_{i,j}$ and $x_{i+1,j}$ get different colors. Finally, on the first and the last copies of $K_p$, adjacent vertices are $x_{1,j}$ of color $j$ and $x_{q, p-1-j}$ of color $p-1-j$. Since $p$ is even $j \not= p-1-j$ for any $j$. Figure \ref{fig:graphcolor4odd} gives the assigned 4-coloring for $H_{4,q}$ in this case.

\noindent{\it Case 2: $p$ and $q$ are both even}. For $1 \le i \le q$ and $i$ is odd, assign to $x_{i,j}$ color $j$ for all $j = 0, \dots, p-1$. For $1 \le i \le q$ and $i$ is even, assign to $x_{i,j}$ the color $p+1-j$. Again, the vertices on each copy of $K_p$ get different colors. Also, since $p$ is even $j \not= p+1-j$, adjacent vertices on consecutive copies of $K_p$ also get different colors. On the first and the last copies of $K_p$, adjacent vertices are $x_{1,j}$ of color $j$ and $x_{q,p-1-j}$ of color $j+2$, and we have $j \not\equiv j+2 ~(\!\mod p)$. Figure \ref{fig:graphcolor4} gives the assigned 4-coloring for $H_{4,q}$ in this case.

\begin{figure}[h!]
\centering
\includegraphics[width=3.5in]{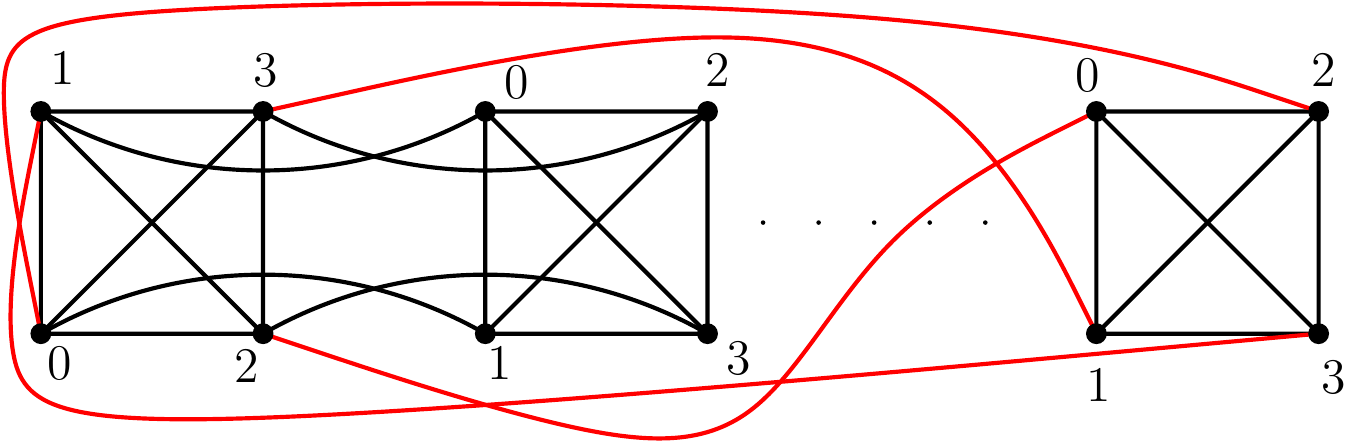}
\caption{A 4-coloring for $H_{4,q}$ when $q$ is even.}\label{fig:graphcolor4}
\end{figure}

\noindent{\it Case 3: $p$ is odd and $q$ is even}. For $1 \le i \le q$ and $i$ is odd, assign to $x_{i,j}$ color $j$ for all $j = 0, \dots, p-1$. For $1 \le i \le q$ and $i$ is even, we assign the colors to $x_{i,j}$s as follows: first, we assign to $x_{i,j}$ color $p-j$, for $j = 0, \dots, p-1$, and then we switch the colors of $x_{i,0}$ and $x_{i,\frac{p+1}{2}}$ (i.e., the vertex $x_{i,0}$ now has color $\frac{p-1}{2}$ and the vertex $x_{i, \frac{p+1}{2}}$ now has color 0). Again, the vertices on each copy of $K_p$ get different colors. On consecutive copies of $K_p$, since $j \not\equiv p-j ~(\!\mod p)$ unless $j = 0$, together with the color switching between $x_{i,0}$ and $x_{i,\frac{p+1}{2}}$, it can be seen that adjacent vertices get different colors. On the first and the last copies of $K_p$, adjacent vertices are $x_{1,j}$ of color $j$ and $x_{q,p-1-j}$ of colors $j+1 \not\equiv j$, except when $j = p-1$ or $j = \frac{p-3}{2}$. Finally, $x_{1,p-1}$ and $x_{q,0}$ are adjacent and of colors $p-1 \not\equiv \frac{p-1}{2}$, while $x_{1,\frac{p-3}{2}}$ and $x_{q, \frac{p+1}{2}}$ are adjacent and of colors $\frac{p-3}{2} \not\equiv 0$ (this is where we make use of the hypothesis that $p \ge 4$). Figure \ref{fig:graphcolor5} gives the assigned 5-coloring for $H_{5,q}$ in this case.

\begin{figure}[h!]
\centering
\includegraphics[height=1.5in]{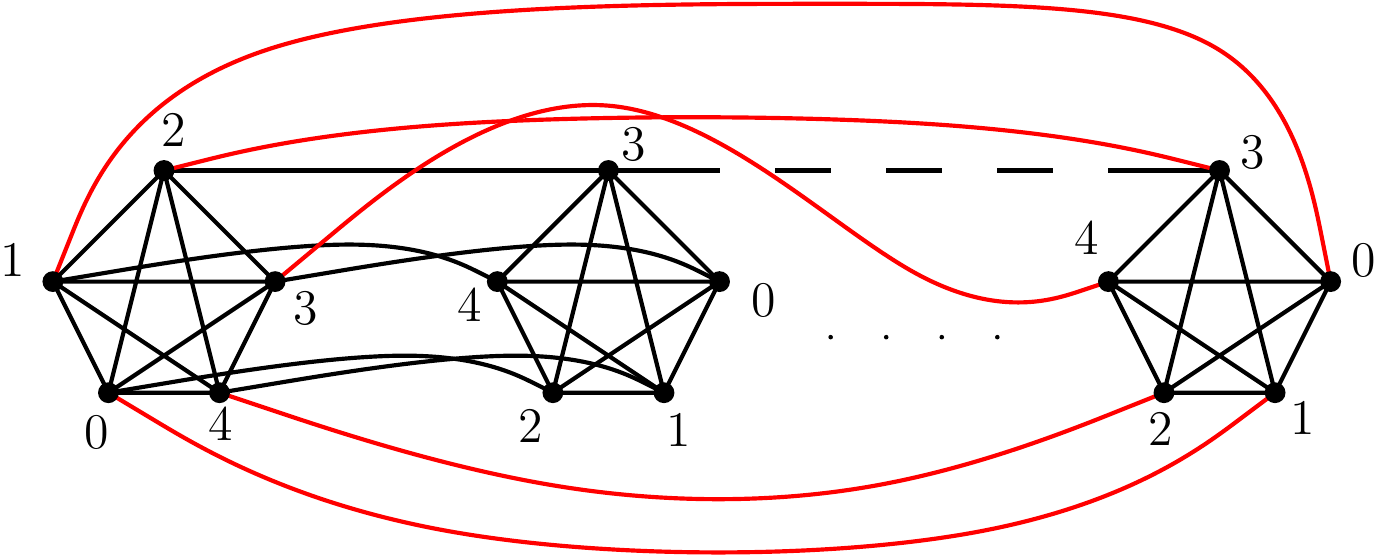}
\caption{A 5-coloring for $H_{5,q}$ when $q$ is even.}\label{fig:graphcolor5}
\end{figure}

\noindent{\it Case 4: $p$ and $q$ are both odd}. For $1 \le i < q-1$ and $i$ is odd, assign to $x_{i,j}$ color $j$ for all $j=0, \dots, p-1$. For $1 \le i \le q-1$ and $i$ is even, assign to $x_{i,j}$ color $j-1$ for all $j=0, \dots, p-1$. Finally, we assign the colors to $x_{q,j}$s as follows: first, we assign to $x_{q,j}$ color $p-j$, for $j=0, \dots, p-1$, and then we switch the colors of $x_{q,0}$ and $x_{q,\frac{p+1}{2}}$ (i.e., the vertex $x_{q,0}$ now has color $\frac{p-1}{2}$ and the vertex $x_{q,\frac{p+1}{2}}$ now has color 0). Clearly, vertices on each copy of $K_p$ get different colors, and adjacent vertices on consecutive copies of $K_p$ (except the last one) get different colors. On the $(q-1)$st and the $q$th copies of $K_p$, adjacent vertices are $x_{q-1,j}$ of color $j-1$ and $x_{q,j}$ of color $p-j$, except exactly when $j=0$ or $j=\frac{p+1}{2}$ due to the color switch. It can be seen that $j-1 \not\equiv p-j$ for all $j \not= \frac{p+1}{2}$. When $j = \frac{p+1}{2}$ the colors of $x_{q-1,\frac{p+1}{2}}$ and $x_{q,\frac{p+1}{2}}$ are $\frac{p-1}{2} \not\equiv 0$. For adjacent vertices between the $q$th and the first copies of $K_p$, the argument follows from the last part of that of Case 3 (and again, we shall need the condition that $p \ge 4$). Figure \ref{fig:graphcolor5odd} gives the assigned 5-coloring for $H_{5,q}$ in this case.

\begin{figure}[h!]
\centering
\includegraphics[height=1.5in]{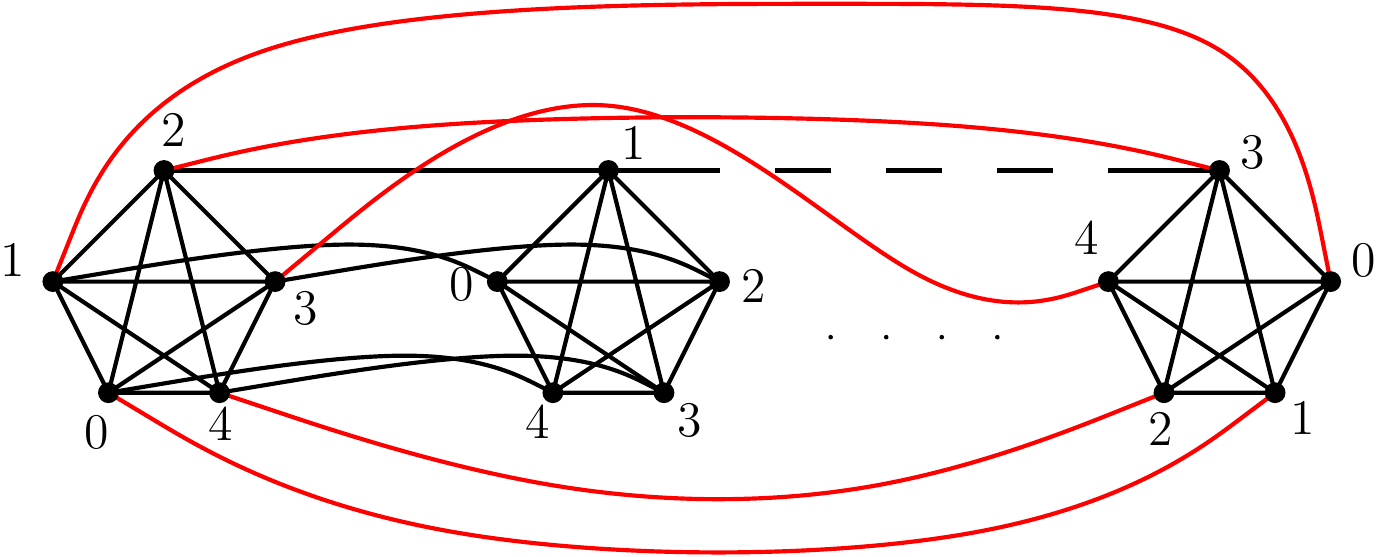}
\caption{A 5-coloring for $H_{5,q}$ when $q$ is odd.}\label{fig:graphcolor5odd}
\end{figure}

\end{proof}


\end{document}